\documentclass[a4paper,twoside]{amsart}
\usepackage{amsmath,amsfonts,amsthm,mathrsfs}
\usepackage[unicode]{hyperref}

\usepackage[alphabetic,initials]{amsrefs}

\usepackage{amssymb}
\usepackage{enumitem}

\usepackage{chngcntr}

\makeatletter
\let\@@seccntformat\@seccntformat
\renewcommand*{\@seccntformat}[1]{%
  \expandafter\ifx\csname @seccntformat@#1\endcsname\relax
    \expandafter\@@seccntformat
  \else
    \expandafter
      \csname @seccntformat@#1\expandafter\endcsname
  \fi
    {#1}%
}
\newcommand*{\@seccntformat@subsection}[1]{%
  \textbf{\csname the#1\endcsname.}
}
\makeatother

\makeatletter
\let\@paragraph\paragraph
\renewcommand*{\paragraph}[1]{%
	\vspace{0.3\baselineskip}%
	\@paragraph{\textit{#1}}%
}
\makeatother

\counterwithin{equation}{subsection}
\counterwithout{subsection}{section}
\counterwithin{figure}{subsection}

\newtheorem{theorem}[equation]{Theorem}
\newtheorem*{theorem*}{Theorem}
\newtheorem{lemma}[equation]{Lemma}
\newtheorem*{lemma*}{Lemma}

\newtheorem*{proposition*}{Proposition}

\theoremstyle{definition}

\newtheorem*{definition*}{Definition}
\theoremstyle{remark}

\newtheorem*{example*}{Example}
\newtheorem*{note}{Note}

\theoremstyle{plain}

\makeatletter
\let\old@caption\caption
\renewcommand*{\caption}[1]{%
	\setcounter{figure}{\value{equation}}%
	\stepcounter{equation}%
	\old@caption{#1}\relax%
}
\makeatother

\newcommand{\decal}[1]{\lbrack #1 \rbrack}

\newcommand{\ltriangle}[4][]%
{\begin{diagram}[#1]%
	{#2} &\rTo& {#3} &\rTo& {#4} &\rTo& {#2 \decal{1}}%
\end{diagram}}

\newcommand{\pr}{\mathit{pr}}

\newcommand{\abs}[1]{\lvert #1 \rvert}

\newcommand{\tensor}{\otimes}

\newcommand{\Hom}{\operatorname{Hom}}

\newcommand{\ZZ}{\mathbb{Z}}
\newcommand{\QQ}{\mathbb{Q}}
\newcommand{\RR}{\mathbb{R}}
\newcommand{\CC}{\mathbb{C}}

\newcommand{\PPn}[1]{\mathbb{P}^{#1}}

\DeclareMathOperator{\id}{id}

\newcommand{\define}[1]{\emph{#1}}

\newcommand{\shf}[1]{\mathscr{#1}}
\newcommand{\OX}{\shf{O}_X}
\newcommand{\OmX}[1]{\Omega_X^{#1}}

\def\overbar#1#2#3{{%
	\setbox0=\hbox{\displaystyle{#1}}%
	\dimen0=\wd0
	\advance\dimen0 by -#2 
	\vbox {\nointerlineskip \moveright #3 \vbox{\hrule height 0.3pt width \dimen0}%
		\nointerlineskip \vskip 1.5pt \box0}%
}}

\usepackage{tikz}
\usetikzlibrary{matrix,arrows}
\newlength{\myarrowsize} 
\pgfarrowsdeclare{nice}{nice}{
	\setlength{\myarrowsize}{0.6pt}
	\addtolength{\myarrowsize}{.5\pgflinewidth}
	\pgfarrowsleftextend{-4\myarrowsize-.5\pgflinewidth} 
	\pgfarrowsrightextend{.5\pgflinewidth}
}{
	\setlength{\myarrowsize}{0.6pt} 
  	\addtolength{\myarrowsize}{.5\pgflinewidth}  
	\pgfsetdash{}{0pt} 
	\pgfsetroundjoin 
	\pgfsetroundcap 
	\pgfpathmoveto{\pgfpointorigin} 
	\pgfpatharc{-110}{-170}{4\myarrowsize} 
	\pgfusepathqstroke
	\pgfpathmoveto{\pgfpointorigin} 
	\pgfpatharc{110}{170}{4\myarrowsize} 
	\pgfusepathqstroke
}
\pgfarrowsdeclaredouble[-\pgflinewidth]{onto}{onto}{nice}{nice}

\newcommand{\Picb}{\overline{\operatorname{Pic}}\vphantom{\operatorname{Pic}}^0}
\newcommand{\Zn}[1]{\ZZ_{#1}}
\newcommand{\Db}{D^b}
\DeclareMathOperator{\Coh}{Coh}
\DeclareMathOperator{\Pic}{Pic}
\newcommand{\Ah}{\Pic^0(A)}

\newcommand{\shP}{\mathcal{P}}

\begin{document}

\title{The fundamental group is not a derived invariant}

\author{Christian Schnell}

\address{Department of Mathematics, Statistics \& Computer Science \\
University of Illinois at Chicago \\
851 South Morgan Street \\
Chicago, IL 60607}

\thanks{The author is partially supported by NSF grant DMS-1100606.}

\email{cschnell@math.uic.edu}

\subjclass[2000]{}



\begin{abstract}
We show that the fundamental group is not invariant under derived equivalence of
smooth projective varieties.
\end{abstract}

\maketitle

\subsection{Acknowledgements}

This short paper is based on my talk at the conference \emph{Derived Categories Tokyo
2011}, and I am very grateful to Yuujiro Kawamata and Yukinobu Toda for inviting me to
the conference. I also thank Lev Borisov for suggesting the example to me in the
first place, and Andrei C\u{a}ld\u{a}raru and Sukhendu Mehrotra for useful discussions.

Some time after giving the talk, I discovered that Anthony Bak \cite{Bak} already has
a preprint on \texttt{arXiv} in which he obtains the same result. In fact, his proof
is more concrete, and therefore more useful for doing calculations with the derived
equivalence. My apology for nevertheless writing this note is that the proof given
here is different and, by relying on the theorem of Bridgeland and Maciocia
\cite{BM}, a little bit shorter than Bak's.

\subsection{Introduction}

For a smooth complex projective variety $X$, we denote by $\Db(X) = \Db \Coh(X)$ the
bounded derived category of coherent sheaves on $X$. Recall that two smooth
projective varieties $X$ are $Y$ are said to be \define{derived equivalent} if
$\Db(X) \simeq \Db(Y)$ as $\CC$-linear triangulated categories. We sometimes write $X
\sim Y$ to indicate that $X$ and $Y$ are derived equivalent.

From the work of Bondal, Orlov, C\u{a}ld\u{a}raru, Kawamata, and others, it is known
that many of the basic invariants of algebraic varieties are preserved under derived
equivalence.  These include the dimension, the Kodaira dimension, the canonical ring,
and the order of the canonical class. It has also been conjectured that the Hodge
structure on the cohomology with rational coefficients is a derived invariant, in the
sense that if $X \sim Y$, then one should have $H^k(X, \QQ) \simeq H^k(Y, \QQ)$ as
rational Hodge structures, for every $k \in \ZZ$. In particular, it is expected that
the Hodge numbers
\[
	h^{p,q}(X) = \dim_{\CC} H^{p,q}(X) = \dim_{\CC} H^q \bigl( X, \OmX{p} \bigr)
\]
are invariant under derived equivalences.

In joint work with Mihnea Popa \cite{d-equivalence}, we showed that if $X$ and $Y$ are
derived equivalent, then $H^1(X, \QQ) \simeq H^1(Y, \QQ)$ as rational Hodge
structures; in geometric terms, this means that the two Picard varieties $\Pic^0(X)$ and
$\Pic^0(Y)$ are isogenous abelian varieties. Ignoring the choice of basepoint,
\[
	H^1(X, \QQ) \simeq \Hom_{\ZZ} \bigl( \pi_1(X), \QQ \bigr),
\]
and so our result naturally leads to the question whether the fundamental group
$\pi_1(X)$ is itself a derived invariant. The point of this paper is to show that
this is not the case. 

More precisely, I will describe an example of a simply connected Calabi-Yau threefold
$X$, with a nontrivial free action by a finite group $G$, such that the quotient
$X/G$ is derived equivalent to $X$. Since $\pi_1(X/G) = G$, while $\pi_1(X) = \{1\}$,
this means that neither the fundamental group nor the property of being simply
connected are preserved under derived equivalence.

\subsection{A related problem}

Before continuing, I should point out that this result is connected to a larger
question raised by Daniel Huybrechts and Marc Nieper-Wisskirchen, about
derived equivalences of varieties with trivial first Chern class. To set up some
notation, suppose that $X$ is a smooth projective variety whose first Chern class
$c_1(X)$ is zero as an element of  $H^2(X, \RR)$. By Yau's theorem, $X$ admits a
Ricci-flat K\"ahler metric; by studying the holonomy of this metric, Bogomolov and
Beauville \cite{Beauville} have shown that a finite \'etale cover $X' \to X$
decomposes into a finite product
\[
	X' \simeq A \times \prod_i Y_i \times \prod_j Z_j
\]
with $A$ an abelian variety, each $Y_i$ a simply connected Calabi-Yau manifold of
dimension at least three, and each $Z_j$ a holomorphic symplectic manifold. 
Huybrechts and Nieper-Wisskirchen ask whether the structure of this decomposition is
invariant under derived equivalences \cite{HNW}*{Question~0.2}.

In the special case of Calabi-Yau threefolds, the question becomes the following:
Suppose that $X$ is a simply connected Calabi-Yau threefold, and that $Y \sim X$.
Because of the example in this paper, we cannot expect $Y$ to be simply
connected. On the other hand, the first Chern class of $Y$ is also trivial, and so a
finite \'etale cover of $Y$ must be of one of the following three types:
\begin{enumerate}
\item A simply connected Calabi-Yau threefold.
\item An abelian threefold.
\item The product of an elliptic curve and a K3-surface.
\end{enumerate}
Note that there are examples of finite quotients of abelian threefolds (or products
of an elliptic curve and a K3-surface) with trivial canonical bundle and zero first
Betti number; a partial classification may be found in \cite{OS}.  Nevertheless, it
seems likely that such varieties cannot be derived equivalent to a simply connected
Calabi-Yau threefold. A proof of this would be a useful step towards answering the
general question of Huybrechts and Nieper-Wisskirchen.

\subsection{The example}

Let us now turn to the description of the example, which was suggested to me by Lev
Borisov. The Calabi-Yau threefold in question is one of a class of such varieties
constructed by Mark Gross and Sorin Popescu \cite{GP}, and has to do with
(1,8)-polarized abelian surfaces.

We shall begin by recalling their construction. Let $(A, L)$ be a 
(1,8)-polarized abelian surface. In other words, suppose that $A$ is an abelian
surface, and that $L$ is an ample line bundle on $A$ such that the isogeny
\[
	\varphi_L \colon A \to \Ah, \quad 
		a \mapsto t_a^{\ast}(L) \tensor L^{-1},
\]
has kernel isomorphic to $\Zn{8} \times \Zn{8}$. 
One can show that $L$ is then
automatically very ample with $h^0(L) = 8$; by the Riemann-Roch theorem, it follows
that we have $8 = \chi(L) = L^2/2$, which gives $L^2 = 16$. The line bundle therefore
embeds $A$ as a surface of degree $16$ into $\PPn{7}$, and it
is possible to choose the coordinates on the projective space in such a way that the action of
$G = \Zn{8} \times \Zn{8}$ on $\PPn{7}$ is given by the formulas
\begin{align*}
	\sigma (x_0 \colon x_1 \colon \dotsb \colon x_6 \colon x_7) 
		&= (x_1 \colon x_2 \colon \dotsb \colon x_7 \colon x_0)  \\
	\tau (x_0 \colon x_1 \colon \dotsb \colon x_6 \colon x_7)
		&= (x_0 \colon \zeta x_1 \colon \dotsb \colon \zeta^6 x_6 \colon \zeta^7 x_7).
\end{align*}
Here $\sigma$ and $\tau$ denote the two natural generators of the group $G$, and
$\zeta$ is a primitive eighth root of unity.

The idea of Gross and Popescu is to look at quadrics in $\PPn{7}$ that contain the
image of $A$. Provided that the pair $(A, L)$ is general in moduli, they show that
the space of such quadrics has dimension four, and that it is generated by the four
quadrics $f$, $\sigma f$, $\sigma^2 f$, and $\sigma^3 f$, where
\[
	f = y_1 y_3 (x_0^2 + x_4^2) - y_2^2 (x_1 x_7 + x_3 x_5) 
		+ (y_1^2 + y_3^2) x_2 x_6,
\]
and $y \in \PPn{2}$ is a general point. The intersection 
\[
	V_{8,y} = Z(f) \cap Z(\sigma f) \cap Z(\sigma^2) f \cap Z(\sigma^3 f)
\]
is then a threefold on which the group $G$ acts freely. Conversely, if we assume that
$y \in \PPn{2}$ is chosen sufficiently general, $V_{8,y}$ will be a complete
intersection of dimension three which is smooth except for $64$ ordinary double points,
the $G$-orbit of the point $(0 \colon y_1 \colon y_2 \colon y_3 \colon 0 \colon -y_3 \colon
-y_2 \colon -y_1)$. There is always a one-dimensional family of $(1,8)$-polarized
abelian surfaces contained in $V_{8,y}$, and every member of the family passes
through the $64$ distinguished points.

Gross and Popescu discovered that $V_{8,y}$ admits two small resolutions $V_{8,y}^1$
and $V_{8,y}^2$, both Calabi-Yau threefolds. The original abelian surface $A$ is a
Weil divisor on $V_{8,y}$ that is not Cartier; blowing up $A$ produces a small
resolution $V_{8,y}^2 \to V_{8,y}$. Using the Lefschetz theorem and adjunction, one
can easily show that $V_{8,y}^2$ is a simply connected Calabi-Yau threefold. The $64$
exceptional curves can be flopped simultaneously to produce another simply connected
Calabi-Yau threefold $V_{8,y}^1$, and Gross and Popescu compute that 
\[
	h^{1,1}(V_{8,y}^1) = h^{1,2}(V_{8,y}^1) = 2.
\]
In fact, the Picard group of $V_{8,y}^1$ is generated (modulo torsion) by the classes
of two divisors: the strict transform $A$ of the original abelian surface, and the
preimage $H$ of a hyperplane section. They satisfy 
\[
	H^3 = 16, \quad H^2 \cdot A = 16, \quad H \cdot A^2 = 0, \quad A^3 = 0.
\]

The reader can find a concise summary of all the properties of the Calabi-Yau
threefolds $V_{8,y}^1$ and $V_{8,y}^2$ in \cite{GPav}. We shall only list those that
are needed below.
\begin{enumerate}
\item The linear system $\abs{A}$ is one-dimensional, and the resulting morphism
\[
	p \colon V_{8,y}^1 \to \PPn{1}
\]
is an abelian surface fibration with exactly $64$ sections. The images of these
sections are the $64$ exceptional curves of the flop. 
\item Every smooth fiber of $p$ is a (1,8)-polarized abelian surface, with
polarization induced by the restriction of the line bundle $\OX(H)$. The intersection
with the images of the $64$ sections is precisely the kernel of the polarization.
\item There are exactly eight singular fibers, each of them an elliptic
translation scroll. Such a scroll is obtained from an elliptic normal curve
$E$ in $\PPn{7}$ by fixing a point $e \in E$, and letting $T_e(E)$ be the union of
all lines through $x$ and $x+e$, for $x \in E$. It is not hard to see that $T_e(E)$
is singular precisely along the elliptic curve $E$, and that $E \times \PPn{1}$ is a
resolution of singularities.
\item The group $G$ acts freely on $V_{8,y}^1$, and the $64$ sections form a single
$G$-orbit.  On each smooth fiber, the action of $G$ is the natural action by the
kernel of the polarization; on each singular fiber, the action of $G$ is the natural
action by the group of $8$-torsion points of the elliptic curve.
\end{enumerate}

Now let $X = V_{8,y}^1$. Since $G = \Zn{8} \times \Zn{8}$ acts freely on $X$, the
quotient $X/G$ is again a smooth projective variety with trivial canonical bundle.
The following theorem is the main result of this paper.

\begin{theorem} \label{thm:main}
The two varieties $X$ and $X/G$ are derived equivalent.
\end{theorem}

Since $X$ is simply connected, while the quotient $X/G$ has fundamental group
isomorphic to $G$, it follows that the fundamental group is not invariant under
derived equivalences.

\begin{note}
For reasons coming from physics, Mark Gross and Simone Pavanelli \cite{GPav}
conjecture that the quotient of $X$ by one of the two $\Zn{8}$-factors of $G$ should
be the mirror manifold of $X$, and that $X/G$ should be the mirror of the mirror.
Homological mirror symmetry would therefore predict that $\Db(X) \simeq \Db(X/G)$.
\end{note}

\subsection{Proof of the theorem}

We now describe one possible proof of Theorem~\ref{thm:main}. An earlier proof, more
concrete but also slightly longer, may be found in the preprint by Anthony Bak
\cite{Bak}.

To explain the basic idea, let us consider one of the smooth fibers $A$ of the
morphism $p \colon X \to \PPn{1}$; it is a $(1,8)$-polarized abelian surface, with
polarization $L$ induced by the restriction of $H$. The group $G$ acts on $A$, and
the $64$ points in the kernel of the isogeny
\[
	\varphi_L \colon A \to \Ah
\]
form a single $G$-orbit. Consequently, $\varphi_L$ gives rise to an isomorphism $A/G
\simeq \Ah$. Moreover, once we choose one of the $64$ points in the kernel as the
unit element on $A$, there is a well-defined Poincar\'e line bundle
on $A \times \Ah$, and the associated Fourier-Mukai transform induces an equivalence
$\Db(A) \simeq \Db(\Ah)$. In our proof, we shall generalize these observations by (1)
using the theorem of Bridgeland and Maciocia to prove that $X$ is derived equivalent
to the compactified relative Picard scheme $M$, and (2) showing that $M$ is actually
isomorphic to $X/G$.

We begin by introducing the space $M$. Let $s_0$ be one of the $64$ sections of $p
\colon X \to \PPn{1}$. The general fiber of $p$ is an abelian surface, and the eight
singular fibers are elliptic translation scrolls, and therefore reduced and
irreducible. Thus it makes sense to consider the compactified relative Picard scheme
\[
	M = \Picb(X / \PPn{1})
\]
defined by Altman and Kleiman \cite{AK}. For any smooth fiber $A$ of $X \to \PPn{1}$,
the corresponding fiber of $M \to \PPn{1}$ is also smooth and isomorphic to $\Ah$.
In general, the compactified relative Picard scheme may be singular, and may fail to
be a fine moduli space because the ambiguity in normalizing the Poincar\'e bundle
can prevent the existence of a universal sheaf. But in our case, everything works
out nicely.

\begin{lemma} \label{lem:step1}
$M$ is a nonsingular projective Calabi-Yau threefold. Moreover, a universal sheaf
exists on $M \times X$, and induces an equivalence $\Db(M) \simeq \Db(X)$.
\end{lemma}

\begin{proof}
To begin with, Sawon \cite{Sawon}*{Lemma~8} has shown that $M$ is projective,
because it is an irreducible component of Simpson's moduli space of stable rank-one
torsion-free sheaves on $X$. Next, the existence of a section $s_0$ implies that
there is a universal sheaf on $M \times_{\PPn{1}} X \subseteq M \times X$. Indeed,
because $X$ is nonsingular, the image of $s_0$ has to lie inside the smooth locus of
$p$, and so we can apply \cite{AK}*{Theorem~3.4} to obtain the existence of a
universal sheaf.  In particular, $M$ is a fine moduli space.  

The theorem of Bridgeland and Maciocia \cite{BM}*{Theorem~1.2} now allows us to
conclude that $M$ is also a nonsingular Calabi-Yau threefold, and that the universal
sheaf induces an equivalence between the derived categories of $M$ and $X$.
\end{proof}

The remainder of the proof consists in showing that $M$ is, in fact, isomorphic to
the quotient $X / G$. Our argument is based on the observation, explained above, that
the smooth fibers of $X/G$ and $M$ are canonically isomorphic. The main issue is to
extend this isomorphism to the singular fibers.

We begin by constructing a rational map from $X/G$ to $M$, using the universal
property of the fine moduli space $M$. The idea is the following: Let $(A, L)$
be a polarized abelian variety, and let $\shP_A$ denote the normalized Poincar\'e
bundle on $A \times \Ah$. The pullback of $\shP_A$ under the morphism
\[
	\id \times \varphi_L \colon A \times A \to A \times \Ah
\]
satisfies
\begin{equation} \label{eq:Poincare}
	(\id \times \varphi_L)^{\ast} \shP_A \simeq
		m^{\ast} L \tensor \pr_1^{\ast} L^{-1} \tensor \pr_2^{\ast} L^{-1},
\end{equation}
where $m \colon A \times A \to A$ is the multiplication on $A$. This allows us to
describe the morphism $\varphi_L$ to the moduli space $\Ah$ in terms of a line bundle
on $A \times A$.

To extend this construction to $X$, let $B \subseteq \PPn{1}$ be the complement of
the eight singular values of $p$, and set $U = p^{-1}(B)$. Then $p \colon U \to B$ is
smooth, and our choice of section $s_0$ makes it into a group scheme over $B$. We
also denote the multiplication morphism by $m \colon U \times_B U \to U$.

\begin{lemma} \label{lem:step2}
There is a morphism $f \colon U/G \to M$, commuting with the projections to
$\PPn{1}$, whose restriction to any smooth fiber agrees with the natural isomorphism
$A/G \to \Ah$ induced by the $(1,8)$-polarization $\OX(H) \vert_A$. 
\end{lemma}

\begin{proof}
We begin by constructing a morphism $U \to M$ whose restriction to every smooth fiber
$A$ agrees with the natural morphism $\varphi_L \colon A \to \Ah$ induced by the
polarization $L = \OX(H) \vert_A$. By the universal property of $M$, it suffices to
construct a line bundle on the product $U \times_B U$ whose restriction to $A \times
A$ is isomorphic to the pullback of the Poincar\'e bundle $\shP_A$. But clearly
\[
	m^{\ast} \OX(H) \tensor \pr_1^{\ast} \OX(-H) \tensor \pr_2^{\ast} \OX(-H)
\]
is such a line bundle, by virtue of \eqref{eq:Poincare}. The same identity shows that
this line bundle is invariant under the action of $G$ on the second factor of $U
\times_B U$, and therefore descends to a line bundle on $U \times_B (U/G)$ whose
restriction to $A \times (A/G)$ equals the pullback of $\shP_A$. The universal
property of $M$ now gives us the desired morphism $f \colon U/G \to M$.
\end{proof}

In particular, $f$ is an isomorphism onto its image, and so $X/G$ and $M$ are 
birational. This is already sufficient to conclude that $\pi_1(M) \simeq G$; but in
fact, we can use the geometry of both varieties to show that they must be isomorphic.

\begin{lemma} \label{lem:step3}
$f \colon U/G \to M$ extends to an isomorphism $X/G \simeq M$.
\end{lemma}

\begin{proof}
We consider $f$ as a birational map from $X/G$ to $M$. Since both are smooth
Calabi-Yau threefolds, any birational map between $X/G$ and $M$ is either an
isomorphism, or a composition of flops \cite{Kollar}*{Theorem~4.9}. The second
possibility is easily ruled out: Indeed, since $f$ is an isomorphism over $U$, the
exceptional locus is contained in the eight singular fibers; moreover, it must be a
union of rational curves \cite{Kollar}*{Proposition~4.6}. Now each singular
fiber is an elliptic translation scroll, which means that the only rational curves on
it are the one-dimensional family of lines on the scroll. Since these lines cover the
singular fiber, which is a divisor in $X/G$, they cannot be flopped. Consequently,
the birational map $f$ must extend to an isomorphism $X/G \simeq M$.
\end{proof}

\begin{note}
Following the argument \cite{GPav}*{Lemma~1.2}, one can show more generally that
$X/G$ does not admit any flops relative to $\PPn{1}$ at all.
\end{note}

\end{document}